\documentclass[11pt,oneside]{amsart}

\usepackage[utf8]{inputenc}
\usepackage[english]{babel}
\usepackage[T1]{fontenc}
\usepackage{amsfonts}
\usepackage{geometry}
\usepackage{color}
\usepackage{amsthm, amssymb, amsmath}
\usepackage{latexsym}
\usepackage{fancyhdr}
\usepackage{caption}
\usepackage[labelformat=simple]{subcaption}

\usepackage{bm}
\usepackage[bookmarks=true]{hyperref}

\usepackage{tikz}
\tikzstyle{every picture}=[line width=.7pt,minimum size=3pt,every label/.append style={font=\normalsize},label distance=2pt]
\tikzstyle{every node}=[font=\normalsize,circle,draw=black,fill=black,inner sep=0pt,minimum width=1.3pt]

\makeatletter
\newtheorem*{rep@theorem}{\rep@title}
\newcommand{\newreptheorem}[2]{%
\newenvironment{rep#1}[1]{%
 \def\rep@title{#2 \ref{##1}}%
 \begin{rep@theorem}}%
 {\end{rep@theorem}}}
\makeatother

\theoremstyle{plain}
\newtheorem{theorem}{Theorem}[section]
\newtheorem{proposition}[theorem]{Proposition}
\newtheorem{corollary}[theorem]{Corollary}
\newtheorem{lemma}[theorem]{Lemma}
\theoremstyle{definition}

\newtheorem{construction}[theorem]{Construction}
\newtheorem{example}[theorem]{Example}

\newtheorem{question}[theorem]{Question}
\newtheorem*{question*}{Question}
\newtheorem{conjecture}[theorem]{Conjecture}
\newtheorem{remark}[theorem]{Remark}
\newreptheorem{theorem}{Theorem}
\newreptheorem{corollary}{Corollary}
\newreptheorem{proposition}{Proposition}

\date{}

\newcommand{\pd}{{\rm pd}}
\newcommand{\reg}{{\rm reg}}

\newcommand{\rk}{{\rm rank}\ }

\newcommand{\Q}{\mathbb{Q}}

\newcommand{\Z}{\mathbb{Z}}
\newcommand{\N}{\mathbb{N}}

\newcommand{\RP}{\mathbb{RP}}

\newcommand{\mf}{\mathfrak}

\title{Powers of monomial ideals with characteristic-dependent Betti numbers}

\author[D. Bolognini, A. Macchia, F. Strazzanti, V. Welker]{Davide Bolognini, Antonio Macchia, Francesco Strazzanti, Volkmar Welker}

\address{{\small Davide Bolognini, Dipartimento di Ingegneria Industriale e Scienze Matematiche, Universit\`a Politecnica delle Marche, Via Brecce Bianche, 60131 Ancona, Italy}}
\email{{\small davide.bolognini.cast@gmail.com}}

\address{{\small Antonio Macchia, Fachbereich Mathematik und Informatik, Freie Universit\"at Berlin, Arnimallee 2, 14195 Berlin, Germany}}
\email{{\small macchia.antonello@gmail.com}}

\address{{\small Francesco Strazzanti, Dipartimento di Matematica ``Giuseppe Peano'', Universit\`a degli Studi di Torino, Via Carlo Alberto 10, 10123 Torino, Italy}}
\email{{\small francesco.strazzanti@gmail.com}}

\address{{\small Volkmar Welker, Philipps-Universit\"at Marburg, Fachbereich Mathematik und Informatik, 35032 Marburg, Germany}}
\email{{\small welker@mathematik.uni-marburg.de}}

\thanks{The second author was supported by the Deutsche Forschungsgemeinschaft (DFG, German Research Foundation) – project number 454595616. \\
\textbf{Data availibility} Data sharing not applicable to this article as no datasets were generated or analysed during the current study.}

\begin{document}

\maketitle

\begin{center} 
\textit{Dedicated to Professor J\"urgen Herzog on the occasion of his 80th birthday}
\end{center}

\begin{abstract}
We explore the dependence of the Betti numbers of monomial ideals on the characteristic of the field. A first observation is that for a fixed prime $p$ either the $i$-th Betti number of all high enough powers of a monomial ideal differs in characteristic $0$ and in characteristic $p$ or it is the same for all high enough powers.  
In our main results we provide constructions and explicit examples of monomial ideals all of whose powers have some characteristic-dependent Betti numbers or whose asymptotic regularity depends on the field. We prove that, adding a monomial on new variables to a monomial ideal, allows to spread the characteristic dependence to all powers. For any given prime $p$, this produces an edge ideal such that the Betti numbers of all its powers over $\Q$ and over $\Z_p$ are different. Moreover, we show that, for every $r \geq 0$ and $i \geq 3$ there is a monomial ideal $I$ such that some coefficient in a degree $\geq r$ of the Kodiyalam polynomials $\mathfrak P_3(I),\ldots,\mathfrak P_{i+r}(I)$ depends on the characteristic. We also provide a summary of related results and speculate about the behaviour of other combinatorially defined ideals.
\end{abstract}

\bigskip


\section{Introduction}

Betti numbers of minimal free resolutions of ideals in a polynomial ring over a field provide some of the most important invariants of ideals. In general, Betti numbers are very hard to compute and this is still true
if one restricts the question to monomial ideals. However, in this setting there are 
some powerful tools available which 
facilitate the calculation, e.g., Hochster's formula \cite{H77}, the lcm-lattice \cite{GPW99}, Betti splittings \cite{FHV09, B16, BF20}
and in characteristic $0$ even a construction of a minimal free resolution \cite{EMO19}.

Since the monic monomials generating a monomial ideal $I$ do not reveal any
information about the coefficient field of the polynomial ring, monomial
ideals can be considered over any coefficient field. 
It is well known that the Betti numbers of a monomial ideal $I$ in a polynomial ring with coefficients in a field $k$ may depend on the characteristic of $k$.
Probably, the first and simplest example of this phenomenon is the Stanley-Reisner ideal of the triangulation of the real projective plane $\RP^2$, used by Reisner in \cite{R76} to demonstrate the characteristic dependence of the
Cohen-Macaulay property. Here the Betti numbers change in characteristic $2$ compared to any other characteristic. However, this kind of examples have mostly been relegated to illustrate weird behaviours that can occur in the study of resolutions and their algebraic invariants, focusing on the independence of the field, see e.g. \cite{K06} and \cite{DK14}.

In this paper we adopt the opposite perspective, exploring the characteristic dependence of the Betti numbers of monomial ideals and 
in particular how it reverberates in their powers.

Recall that a monomial ideal $I$ has a unique minimal system of monic monomial generators $G(I)$.
Throughout the paper, when we write that we study the \textit{field} or \textit{characteristic dependence} of an invariant for a monomial ideal $I$, 
we mean that we study for different fields $k$ this invariant for the ideal generated by $G(I)$ in the polynomial ring $k[x_1, \dots, x_n]$. For example, we write $\beta_i^k(I)$ for the $i$-th Betti number of $I$ seen as an ideal in $k[x_1, \dots, x_n]$.   

We will be mainly interested in the asymptotic characteristic dependence of Betti numbers for high powers of monomial ideals. 
Recall that, if $I$ is a homogeneous ideal of $k[x_1,\dots,x_n]$, Kodiyalam \cite{K93} proved that for every $1 \leq i \leq n$ there exists a polynomial $\mathfrak{P}^k_i(I)(h)$ such that $\mathfrak{P}^k_i(I)(h)=\beta^k_{i-1}(I^h)$ for $h \gg 0$; we call $\mathfrak{P}^k_i(I)$ the $i$-th {\it Kodiyalam polynomial} of $I$. As a consequence, we observe that either the $i$-th Betti number of all high enough powers of an ideal depends on the characteristic of the field or it does not.

\begin{repproposition}{P.KodiyalamPolynomials}
Let $I$ be a monomial ideal in $k[x_1,\dots,x_n]$ and $p \geq 2$ be a prime number. Then, for every integer $i \geq 0$ there exists $h_i \geq 1$ such that either $\beta^{\Z_p}_i(I^h)=\beta^{\Q}_i(I^h)$ for every $h\geq h_i$ or $\beta^{\Z_p}_i(I^h) \neq \beta^{\Q}_i(I^h)$ for every $h\geq h_i$.
\end{repproposition}

The propagation of characteristic dependence of Betti numbers from the first powers to higher powers is much more mysterious. For instance, 
we show an example of a monomial ideal, indeed an edge ideal, whose Betti numbers are independent of the characteristic, but some Betti numbers of its square depend on the field (see Example \ref{E.edgeIdealWithCharDepSquare}). Using the lcm-lattice for proofs, we provide several examples of squarefree monomial ideals with characteristic-dependent Betti numbers in all powers. One of them is the Stanley-Reisner ideal of a minimal triangulation of the \textit{Klein bottle}.

\begin{reptheorem}{T.KleinBottle}
Let $I = (x_3x_8, x_4x_5, x_6x_7, x_7x_8, x_1x_2x_4, x_1x_3x_4, x_2x_3x_4, x_1x_2x_5, x_2x_3x_5, \break x_1x_4x_6, x_1x_5x_6, x_2x_5x_6, x_1x_2x_7, x_1x_3x_7, x_2x_4x_7, x_3x_5x_7,  x_1x_2x_8, x_1x_5x_8, x_2x_6x_8, x_1x_3x_6, \break x_2x_3x_6, x_4x_6x_8)$ in $k[x_1,\dots,x_8]$ be the Stanley-Reisner ideal of the triangulation of the Klein bottle in Figure \ref{F.KleinBottle}. Then, the Betti numbers of $I^h$ depend on the field for every $h \geq 1$.
\end{reptheorem}

We then turn to the characteristic dependence of the Castelnuovo-Mumford regularity $\reg_k(I)$ of powers of $I$. In terms of Betti numbers this can be seen as the question of whether certain graded Betti numbers are zero and nonzero over different fields. In \cite[Remark 5.3]{MV21} Minh and Vu exhibit a specific edge ideal whose asymptotic regularity depends on the field. We present a general construction that produces a monomial ideal with the same property: it is enough to add a certain power of a new variable $y$ to a monomial ideal whose regularity depends on the field:

\begin{repproposition}{P.asymptoticRegularity}
Let $I \subseteq k[x_1,\dots,x_n,y]$ be a nonzero monomial ideal with generators in the variables $x_1,\dots,x_n$. Suppose that there exists another field $k'$ such that $\reg_{k}(I) \neq \reg_{k'}(I)$. Then, there exists $c \in \mathbb{N}$ such that $\reg_{k} ((I+(y^c))^h) \neq \reg_{k'} ((I+(y^c))^h)$, for $h \gg 0$.
\end{repproposition}

It is also interesting to look for simple constructions that propagate the characteristic dependence of the Betti numbers to all powers. With an argument involving Betti splittings, we prove the following result:

 \begin{reptheorem}{T.Formula}
Let $I$ be a monomial ideal in $k[x_1,\dots,x_n, y_1,\dots,y_r]$, with generators in the variables $x_1,\dots,x_n$. Let $w$ be a monic monomial in the variables $y_1,\dots,y_r$. If $I^h$ has characteristic-dependent Betti numbers for some $h \geq 1$, then the same holds for $(I +(w))^{\ell}$ for every $\ell \geq h$.
\end{reptheorem}

This result has a number of interesting consequences. First, in Corollary \ref{C.pfoldDunceCapCorollary}, for every prime number $p$, we construct an edge ideal, coming from the \textit{p-fold dunce cap} (see Construction \ref{C.pfoldDunceCap}), all of whose powers have different Betti numbers over $\Q$ and $\Z_p$.

Lemma \ref{L.spreadingDependenceToPowers} provides a lower bound on the size of dependencies produced by Theorem \ref{T.Formula}, whereas Lemma \ref{L.SpreadingDependence2} shows that in each power $(I + (y_1,\dots,y_r))^h$ there are at least exponentially many dependencies in $h$.

As a further consequence we show that for the Kodiyalam polynomials the characteristic dependence can be spread over consecutive homological positions:

\begin{reptheorem}{T.Kodiyalam}
For every $i \geq 3$ and for every $r \in \mathbb{N}$, there exists a monomial ideal $I$ such that all the Kodiyalam polynomials $\mathfrak P^k_{3}(I),\mathfrak P^k_{4}(I),\dots,\mathfrak P^k_{i+r}(I)$ have the coefficient at some degree $\geq r$ depending on the characteristic of $k$.
\end{reptheorem}

We conclude with some open questions and extensions to combinatorially defined ideals beyond monomial ideals.
In particular, we provide interesting examples of binomial edge ideals, exhibiting various behaviours with respect to characteristic dependence of the Betti numbers of their first few powers.

\section{Notation and preliminaries}

Let $R=k[x_1, \dots, x_n]$ be the standard graded polynomial ring over a field $k$ and let $I$ be a monomial ideal in $R$. For every $i,j \in \mathbb{N}$, the {\em graded Betti numbers} of $I$, defined as $\beta^k_{i,j}(I)=\dim_k {\rm Tor}^R_i(I,k)_j$, are invariants of the minimal graded free resolution of $I$.
We denote by $\beta^k_i(I)=\sum_{j} \beta^k_{i,j}(I)$ the \textit{$i$-th (total) Betti number} of $I$. If $R$ is standard multigraded, i.e., $\deg(x_i)$ is the $i$-th standard basis vector of $\mathbb{R}^n$, we can define multigraded Betti numbers analogously. In this case, 
if $\alpha = (\alpha_1, \alpha_2, \dots, \alpha_n) \in \mathbb{N}^n$, we set $\mathbf{x}^\mathbf{\alpha} = x_1^{\alpha_1} x_2^{\alpha_2} \cdots x_n^{\alpha_n}$ and denote the corresponding \textit{multigraded Betti number} by $\beta^k_{i,\alpha}(I)$. Throughout the paper we are going to use the term multidegree to refer either to the exponent vector $\alpha$ or to the monomial $m = \mathbf{x}^\alpha$. In the latter case we use the notation $\beta_{i,m}^k(I)$.

Betti numbers encode many important properties of $I$ and their behaviour has been intensively studied in literature. Hilbert's Syzygy Theorem states that $\beta^k_{i}(I)=0$ for $i>n$ and the maximum $i$ such that $\beta^k_{i}(I) \neq 0$ is the {\it projective dimension} of $I$, denoted by $\pd_k(I)$. Another important invariant of $I$ that can be read off from its Betti numbers is the {\it Castelnuovo-Mumford regularity}, defined as $\reg_k(I)=\max\{j-i : \beta^k_{i,j}(I) \neq 0\}$. Whenever it is not important to specify the field $k$, we simply write $\beta_i(I), \pd(I), \reg(I)$.

Even though the field $k$ is involved in the definition of Betti numbers, the degree of influence of the field is not immediately obvious. Indeed, it is well known that the Betti numbers only depends on the characteristic of $k$. Moreover, from the Universal Coefficient Theorem it follows that 
\begin{equation}\label{Eq.universalCoefficientTheorem}
\beta^{\mathbb{Q}}_{i,\alpha}(I) \leq \beta^{\mathbb{Z}_p}_{i,\alpha}(I)
\end{equation}
for every $i \in \mathbb{N}$, $\alpha \in \N^n$ and every prime integer $p$. We redirect the reader to \cite[Proposition 1.3]{K06} for more details. 

When $I$ is a monomial ideal, a useful tool to compute its Betti numbers is its {\it lcm-lattice} $L_I$, introduced in \cite{GPW99}. Let $G(I)$ denote the unique minimal system of monomial generators of $I$. The elements of $L_I$ are the least common multiples of the subsets of $G(I)$ ordered by divisibility. We remark that the minimal element of $L_I$ is $1$, considered as the least common multiple of the
empty set, the atoms are the elements of $G(I)$, and the maximal element is the least common multiple of the elements of $G(I)$.

Given $m \in L_I$, we set $(1,m)_{L_I}=\{m' \in L_I : 1 < m' < m \}$ to be the open interval below $m$ in $L_I$. The {\it order complex} of $(1,m)_{L_I}$ is the abstract simplicial complex whose faces are the chains in $(1,m)_{L_I}$. Identifying $(1,m)_{L_I}$ with its order complex
we can consider the reduced simplicial homology groups
$\widetilde{H}_{\bullet} ((1,m)_{L_I}; k)$.
With this notation, \cite[Theorem 2.1]{GPW99} shows that the multigraded Betti numbers of $I$ are given by
\[
\beta^k_{i,m}(I)= \dim \widetilde{H}_{i-1} ((1,m)_{L_I}; k)
\]
for every $m \in L_I$ and by $\beta^k_{i,m}(I)=0$ if $m \notin L_I$.

For further details about simplicial complexes, Stanley-Reisner ideals and their combinatorics we refer to \cite{HH11}.

Another technique to compute the Betti numbers of a monomial ideal $I$ is the so-called {\it Betti splitting}, see \cite{FHV09, B16, BF20}.
Let $I,J,K$ be monomial ideals in $R=k[x_1,\dots,x_n]$ such that $G(I)$ is the disjoint union of $G(J)$ and $G(K)$. We say that $I=J+K$ is a {\it Betti splitting} of $I$ if
$$\beta^k_{i}(I)=\beta^k_{i}(J)+\beta^k_{i}(K)+\beta^k_{i-1}(J \cap K)$$
for every $i \in \mathbb{N}$. Given the short exact sequence $$0 \rightarrow J \cap K \rightarrow J \oplus K \rightarrow J+K \rightarrow 0,$$ we have an induced long exact sequence of $\mathrm{Tor}$ modules, and it is not difficult to show that $I=J+K$ is a Betti splitting of $I$ if and only if the induced maps
\[
\mathrm{Tor}^R_i(J \cap K,k) \rightarrow \mathrm{Tor}^R_i(J,k) \oplus \mathrm{Tor}^R_i(K,k)
\]
are zero for every $i \in \mathbb{N}$, see \cite[Proposition 2.1]{FHV09}. Considering graded or multigraded maps, one can define Betti splittings in the graded or multigraded setting.

\section{Asymptotic behaviour}\label{S.asymptoticBehaviour}

Since the dependence on the field of the Betti numbers of an ideal is only through its characteristic, we will compare the Betti numbers over $\Q$ and $\Z_p$ for some prime integer $p \geq 2$. 



\subsection{General facts} 
We start by studying the asymptotic characteristic dependence for monomial ideals.

\begin{proposition}\label{P.KodiyalamPolynomials}
Let $I$ be a monomial ideal in $k[x_1,\dots,x_n]$ and $p \geq 2$ be a prime number. Then, for every integer $i \geq 0$ there exists $h_i \geq 1$ such that either $\beta^{\Z_p}_i(I^h)=\beta^{\Q}_i(I^h)$ for every $h\geq h_i$ or $\beta^{\Z_p}_i(I^h) \neq \beta^{\Q}_i(I^h)$ for every $h\geq h_i$.
\end{proposition}

\begin{proof}
If the Kodiyalam polynomials $\mf P_{i+1}^{\Z_p}(I)(h)$ and $\mf P_{i+1}^{\Q}(I)(h)$ are equal, then clearly\break $\beta_i^{\Z_p}(I^h)=\beta_i^{\Q}(I^h)$ for every $h\gg 0$. Otherwise, since the polynomial $\mf P_{i+1}^{\Z_p(I)} - \mf P_{i+1}^{\Q}(I)$ has a finite number of roots, the equality $\beta_i^{\Z_p}(I^h) = \beta_i^{\Q}(I^h)$ holds only for a finite number of integers $h$, and hence $\beta_i^{\Z_p}(I^h)\neq\beta_i^{\Q}(I^h)$ for $h \gg 0$.
\end{proof}

The proof of Proposition \ref{P.KodiyalamPolynomials} works more in general if $I$ is a homogeneous ideal generated by polynomials with integer coefficients which allows one to consider the ideal in the respective polynomial ring over 
any field. Particularly interesting is the case when the coefficients 
are $\pm 1$. For instance, this is the case of binomial edge ideals that we consider in Section \ref{bei}. 

\smallskip

On the other hand, the behaviour of the first few powers of a monomial ideal seems hard to control. For example, let $\Delta$ be the unique (up to simplicial isomorphism) $6$-vertex triangulation $\Delta$ of the real projective plane $\mathbb{RP}^2$ \cite{R76}. Then, the Stanley-Reisner ideal of $\Delta$ is 
\begin{equation}\tag{$\ast$}\label{Eq.ProjectivePlane}
I_\Delta \!=\! (x_1 x_2 x_3,\! x_1 x_2 x_4, x_1 x_3 x_5, x_1 x_4 x_6, x_1 x_5 x_6, x_2 x_3 x_6, x_2 x_4 x_5, x_2 x_5 x_6, x_3 x_4 x_5, x_3 x_4 x_6)
\end{equation}
and its Betti numbers differ over $\Q$ and over $\Z_2$. However, one can check with \textit{Macaulay2} \cite{M2} that this is not the case for $I_{\Delta}^h$ with $h=2,\dots,10$.

There are also cases in which the dependence appears in the second power, even though the resolution of the original ideal does not depend on the field.

\begin{example}\label{E.edgeIdealWithCharDepSquare}
Recall that the \textit{edge ideal} of a graph $G$ is defined by $I(G) = (x_i x_j : \{i,j\} \in E(G))$. Let us consider the graph $G$ whose edge ideal is
\begin{align*}
I(G)=&\ (x_1x_2,x_2x_3,x_2x_4,x_2x_5,x_2x_6,x_2x_{12},x_1x_4,x_1x_6,x_1x_7,x_1x_8,x_2x_{12},x_3x_5, \\
&\ \ x_3x_8, x_3x_{11}, x_3x_{12}, x_4x_5, x_4x_9, x_4x_{10}, x_5x_7, x_5x_9, x_6x_7, x_6x_{10}, x_6x_{11}, x_7x_8, \\
&\ \ x_7x_9, x_7x_{12}, x_8x_{11}, x_9x_{10}, x_9x_{12}, x_{10}x_{11}, x_{10}x_{12}, x_{11}x_{12}).
\end{align*}
Computations with \textit{Macaulay2} show that the Betti numbers of $I(G)$ and of $I(G)^3$ are the same over $\mathbb{Q}$ and over $\Z_2$, whereas $\beta^{\Z_2}_5(I(G)^2) \neq \beta^{\mathbb{Q}}_5(I(G)^2)$. This example also shows that the $5$-th Betti number of the square of an edge ideal may depend on the characteristic of the field. As a consequence there is
no extension of a result by Katzman to powers. The result states that the first six Betti numbers of an edge ideal are characteristic-independent, see \cite[Theorem 3.4 and Corollary 4.2]{K06}. 
\end{example}

\subsection{The Stanley-Reisner ideal of the Klein bottle}

We now show that the Stanley-Reisner ideal of the vertex-minimal triangulation of the Klein bottle in Figure \ref{F.KleinBottle} (see also the top left triangulation in \cite[Figure 18]{C94}) has characteristic-dependent Betti numbers in all powers.

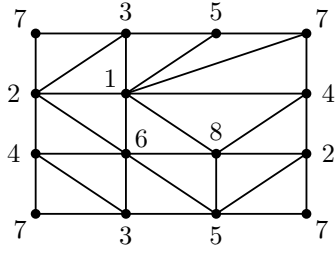
\begin{figure}[ht!]
\centering
\begin{tikzpicture}[scale=0.8]
\node[label={below left:{\small $7$}}] (a) at (0,0) {};
\node[label={below:{\small $3$}}] (b) at (1.5,0) {};
\node[label={below:{\small $5$}}] (c) at (3,0) {};
\node[label={below right:{\small $7$}}] (d) at (4.5,0) {};
\node[label={left:{\small $4$}}] (e) at (0,1) {};
\node[label={above right:{\small $6$}}] (f) at (1.5,1) {};
\node[label={above:{\small $8$}}] (g) at (3,1) {};
\node[label={right:{\small $2$}}] (h) at (4.5,1) {};
\node[label={left:{\small $2$}}] (i) at (0,2) {};
\node[label={above left:{\small $1$}}] (j) at (1.5,2) {};
\node[label={right:{\small $4$}}] (k) at (4.5,2) {};
\node[label={above left:{\small $7$}}] (l) at (0,3) {};
\node[label={above:{\small $3$}}] (m) at (1.5,3) {};
\node[label={above:{\small $5$}}] (n) at (3,3) {};
\node[label={above right:{\small $7$}}] (o) at (4.5,3) {};
\draw (0,0) -- (4.5,0) -- (4.5,3) -- (0,3) -- (0,0);
\draw (0,1) -- (4.5,1) -- (3,0) -- (1.5,1) -- (1.5,0) -- (0,1);
\draw (3,0) -- (3,1) -- (4.5,2) -- (1.5,2) -- (1.5,1) -- (0,2) -- (1.5,2) -- (3,1);
\draw (0,2) -- (1.5,3) -- (1.5,2) -- (3,3);
\draw (1.5,2) -- (4.5,3);
\end{tikzpicture}
\caption{A vertex-minimal triangulation of the Klein Bottle}\label{F.KleinBottle}
\end{figure}

\begin{theorem} \label{T.KleinBottle}
Let $I = (x_3x_8, x_4x_5, x_6x_7, x_7x_8, x_1x_2x_4, x_1x_3x_4, x_2x_3x_4, x_1x_2x_5, x_2x_3x_5, \break x_1x_4x_6, x_1x_5x_6, x_2x_5x_6, x_1x_2x_7, x_1x_3x_7, x_2x_4x_7, x_3x_5x_7,  x_1x_2x_8, x_1x_5x_8, x_2x_6x_8, x_1x_3x_6, \break x_2x_3x_6, x_4x_6x_8)$ in $k[x_1,\dots,x_8]$ be the Stanley-Reisner ideal of the triangulation of the Klein bottle in Figure \ref{F.KleinBottle}. Then, the Betti numbers of $I^h$ depend on the field for every $h \geq 1$.
\end{theorem}

\begin{proof}
It is a simple consequence of the fact that the simplicial homology of
any triangulation of the Klein bottle depends on the characteristic of the coefficient field and Hochster's formula \cite{H77} that the Betti numbers of $I$ depend on the field. In particular, also easily checked using \textit{Macaulay2}, one can verify that $\beta_4^{\Z_2}(I) \neq \beta_4^{\Q}(I)$ and $\beta_5^{\Z_2}(I) \neq \beta_5^{\Q}(I)$.

For every $h \geq 2$, we first show that $\beta_{4,\alpha_h}^{\Z_2}(I) \neq \beta_{4,\alpha_h}^{\Q}(I)$ and $\beta_{5,\alpha_h}^{\Z_2}(I) \neq \beta_{5,\alpha_h}^{\Q}(I)$, where $\alpha_h=(1,1,1,h,h,1,1,1) \in \mathbb N^8$. In order to do this, we define the ideal
\[
J_h=(m \in G(I^h) : m \text{ divides } m_h ),
\]
where $G(I^h)$ is the minimal set of generators of $I^h$ and $m_h=\mathbf{x}^{\alpha_h}=x_1x_2x_3x_4^hx_5^hx_6x_7x_8$.

\textbf{Claim 1.} For every $h \geq 4$,
\[
J_h=x_4x_5 J_{h-1} = (x_4x_5)^{h-3} J_3.
\]
The inclusion $x_4x_5 J_{h-1} \subseteq J_h$ is clear. Conversely, let $m \in G(J_h)$, then $m$ divides $m_h$ and $m=u_1 \cdots u_h$, where $u_i \in G(I)$. Since $\deg(u_i) \geq 2$, it follows that $\deg(m) \geq 2h$. Moreover, $\deg_m(x_i) \leq 1$ for every $i \in \{1,2,3,6,7,8\}$, and hence $x_4^a x_5^b$ divides $m$, with $a+b \geq 2$ (since $\deg(m) \geq 2h \geq 8$). We want to show that $a,b \geq 1$. Assume that $x_5$ does not divide $m$. Thus $a \geq 2$, i.e., $\deg_m(x_4) \geq 2$. Since $u_i \neq x_4x_5$ for every $i$ and $x_4x_5$ is the only generator of $I$ with degree $2$ and divisible by $x_4$, we may assume that $u_1=x_{i_1}x_{i_2}x_4$ and $u_2=x_{i_3}x_{i_4}x_4$, where the indices $i_1,i_2,i_3,i_4$ are pairwise distinct and different from $4$ and $5$. Now, $\deg(u_i) \geq 2$ for every $i=3,\dots,h$. Since $\deg_m(x_i) \leq 1$ for every $i \in \{1,2,3,6,7,8\}$, it follows that $m$ is divisible by at least $2h+1 \geq 9$ pairwise distinct variables, a contradiction. Hence, both $x_4$ and $x_5$ divide $m$. Finally notice that $\frac{m}{x_4x_5} \in J_{h-1}$.

Now consider the polarization $\mathrm{pol}(J_h)$ of $J_h$ in the polynomial ring \break $ k[x_1,\dots,x_8,y_1,\dots,y_{h-1},z_1,\dots,z_{h-1}]$ and the simplicial complex $\Delta_h$ whose Stanley-Reisner ideal is $\mathrm{pol}(J_h)$. Let $\Gamma_h = \Delta_h^\ast$ be the Alexander dual of $\Delta_h$.

\textbf{Claim 2.} The (reduced) homology of $\Gamma_h$ equals the homology of the dual of the triangulation of Figure \ref{F.KleinBottle}. In particular, for every $h \geq 2$,
\begin{gather*}
\widetilde H_3(\Gamma_h,\mathbb Z_2) = \mathbb Z_2, \widetilde H_4(\Gamma_h,\mathbb Z_2) = (\mathbb Z_2)^2, \text{ while}\\
\widetilde H_3(\Gamma_h,\Q) = 0, \widetilde H_4(\Gamma_h,\Q) = \Q.
\end{gather*}

For $h=2,3$, the claim follows by direct computations with \textit{Macaulay2}, while for $h \geq 4$, it follows from Claim 1.

Let us denote by $L_{I^h}$ the lcm-lattice of $I^h$. By \cite[Proposition 2.3]{GPW99}, the lcm-lattice is preserved under polarization and the interval $(1, m_h)_{L_{I^h}}$ is homotopy equivalent to $\Gamma_h$. From \cite[Theorem 2.1]{GPW99}, it then follows that
\[
\beta_{4,\alpha_h}^k(I^h)=\rk \widetilde H_3 (\Gamma_h;k) \text{ and } \beta_{5,\alpha_h}^k(I^h)=\rk \widetilde H_4 (\Gamma_h;k),
\]
hence they are different if $k = \Z_2$ and $k = \Q$.

As a consequence of inequality \eqref{Eq.universalCoefficientTheorem}, 
for every multidegree $\varepsilon \in \mathbb N^8$, $\beta_{4,\varepsilon}^{\Q}(I^h) \leq \beta_{4,\varepsilon}^{\Z_2}(I^h)$ and $\beta_{5,\varepsilon}^{\Q}(I^h) \leq \beta_{5,\varepsilon}^{\Z_2}(I^h)$. In particular, this implies that $\beta_4^{\Z_2}(I^h) \neq \beta_4^{\Q}(I^h)$ and $\beta_5^{\Z_2}(I^h) \neq \beta_5^{\Q}(I^h)$.
\end{proof}

There are six combinatorially distinct $8$-vertex triangulations of the Klein bottle, see \cite[Figure 18]{C94}. By using \textit{Macaulay2}, one can check that for four of these triangulations $\Delta$ the Betti numbers of powers $I_{\Delta}^h$ for small $h \geq 2$ do not depend on the field, while for the other two this is not the case. It follows that the dependence of the Betti numbers of the powers of a monomial ideal is not a topological property, i.e, does not depend only on the homeomorphism type of the simplicial complex. Indeed this example shows that the dependence is influenced by the combinatorics of the triangulation, which in turn governs the divisibility between the generators of the powers.

\subsection{Kimura, Terai, and Yoshida's ideal}\label{S.KimuraTeraiYoshida}

In \cite[Section 6]{KTY09}, Kimura, Terai, and Yoshida consider the following ideal in $k[x_1,\dots,x_{10}]$:
\begin{equation*}
A = (x_1x_2x_8x_9x_{10}, x_2x_3x_4x_5x_{10}, x_5x_6x_7x_8x_{10}, x_1x_4x_5x_6x_9, x_1x_2x_3x_6x_7, x_3x_4x_7x_8x_9).
\end{equation*}
This ideal has $6$ generators of the same degree in $10$ variables and can be obtained from the projective plane according to the construction in \cite[page 76]{KTY09}. Using an argument similar to the one in the proof of Theorem \ref{T.KleinBottle}, one can show that some Betti numbers of $A^h$ depend on the field for every $h \geq 1$. In particular, the multigraded Betti numbers $\beta_{2,\alpha_h}(A^h)$ and $\beta_{3,\alpha_h}(A^h)$ are different over $\Q$ and over $\Z_2$, where $\alpha_h = (h, h, 1, 1, 1, 1, 1, h, h, h) \in \N^{10}$. 

Clearly, for a monomial ideal the zero-th Betti number does not depend on the field and the same holds for the first Betti number, see \cite[Corollary 5.3]{BH95}. However, $\beta_2^{\Z_2}(A^h) \neq \beta_2^{\Q}(A^h)$ and $\beta_3^{\Z_2}(A^h) \neq \beta_3^{\Q}(A^h)$ for every $h \geq 1$. In particular, the Kodiyalam polynomials $\mathfrak P_3^k(A)$ and $\mathfrak P_4^k(A)$ depend on the field.

\subsection{Castelnuovo-Mumford regularity}

By \cite{CHT99} and \cite{K93}, given a homogeneous ideal $I$, the Castelnuovo-Mumford regularity of $I^h$ is asymptotically a linear function in $h$.
Denote by $s_k(I)=\min \{s: \reg_k(I^h)=a_k h+b_k, \text{ for all } h \geq s\}$ the \textit{index of stability of $I$ with respect to $k$}.
In order to prove the next result, we recall the following:

\begin{theorem}[{\cite[Theorem 5.6]{HTT15}}]\label{HTT}
Let $I \subseteq k[x_1,\ldots,x_n]$ and $J \subseteq k[y_1,\ldots,y_r]$ be homogeneous ideals in polynomial rings on disjoint sets of variables such that $\reg_k(I^h)=ah+b$ and $\reg_k(J^h)=ch+d$, for $h \gg 0$. If $c>a$, then
\[
\reg_k((I+J)^h)=c(h+1)+d+\max_{j \leq s_k(I)}\{\reg_k(I^j)-cj\}-1, \text{   for  } h \gg 0.
\]
\end{theorem}

As a consequence, we present a simple construction which produces monomial ideals with the asymptotic regularity of powers depending on the characteristic.

\begin{proposition}\label{P.asymptoticRegularity}
Let $I \subseteq k[x_1,\dots,x_n,y]$ be a nonzero monomial ideal with generators in the variables $x_1,\dots,x_n$. Suppose that there exists another field $k'$ such that $\reg_{k}(I) \neq \reg_{k'}(I)$. Then, there exists $c \in \mathbb{N}$ such that $\reg_{k} ((I+(y^c))^h) \neq \reg_{k'} ((I+(y^c))^h)$, for $h \gg 0$.
\end{proposition}

\begin{proof}
Define integers $$c_k(I)=\mathrm{max}\{a_k+1,\max_{i \leq s_k(I)}\{\reg_k(I^i)\}\}\,,\,c_{k'}(I)= \{a_{k'}+1,\max_{i \leq s_{k'}(I)}\{\reg_{k'}(I^i)\}\},$$ $$\text{ and } 
c=\max\{c_k(I),c_{k'}(I)\}.
$$
Notice, that $\displaystyle{\reg_k((y^c)^h)=\reg_{k'}((y^c)^h)=ch}$, for $h \geq 1$. Since $c>\max \{a_k,a_{k'}\}$, by Theorem \ref{HTT} we have
\[
\reg_k((I+(y^c))^h) = c(h+1) + \max_{j \leq s_k(I)}\{\reg_k(I^j)-cj\} - 1
\]
and
\[
\reg_{k'}((I+(y^c))^h) = c(h+1) + \max_{j \leq s_{k'}(I)}\{\reg_{k'}(I^j)-cj\} - 1
\]
for $h \gg 0$. We claim that $\displaystyle{\max_{j \leq s_k(I)}\{\reg_k(I^j)-cj\}=\reg_k(I)-c}$. In fact, for $2 \leq j \leq s_k(I)$ we have
\[
\reg_k(I^j)-cj \leq c_k(I)-cj \leq c(1-j) \leq -c<\reg_k(I)-c.
\]
Analogously, we get $\mathrm{max}_{j \leq s_{k'}(I)}\{\reg_{k'}(I^j)-cj\}=\reg_{k'}(I)-c$. It follows that
\[
\reg_k((I+(y^c))^h)=ch+\reg_k(I)-1 \text{\quad and \quad} \reg_{k'}((I+(y^c))^h)=ch+\reg_{k'}(I)-1
\]
for $h \gg 0$. This proves the claim.
\end{proof}

\begin{example}\label{E.KatzmanEdgeIdeal}
In \cite[Problem 7.10]{BBH19} the authors ask whether there exist edge ideals $I$ for which the asymptotic linear function $\reg_k(I^h)$, for $h \gg 0$, is characteristic-dependent. Minh and Vu \cite[Remark 5.3]{MV21} answered this question positively, showing that this is the case for the edge ideal of a graph with $18$ vertices.

In \cite[Appendix A]{K06}, Katzman found four non-isomorphic graphs with $11$ vertices whose edge ideal has characteristic-dependent resolution and proved that they are the smallest ones with this property. The edge ideal of one of them is:
\begin{align*}
I(G) = & \ (x_1x_5,x_1x_6,x_1x_8,x_1x_{10},x_2x_5,x_2x_6,x_2x_9,x_2x_{11},x_3x_7,x_3x_8,x_3x_9,x_3x_{11},x_4x_7,\\
\nonumber & \ \ x_4x_8,x_4x_{10},x_4x_{11}, x_5x_8,x_5x_9,x_6x_{10},x_6x_{11},x_7x_9,x_7x_{10},x_8x_{11})
\end{align*}
in $k[x_1,\dots,x_{11}]$. One can check with \textit{Macaulay2} that $\reg_{\Z_2}(I(G)) \neq \reg_{\Q}(I(G))$ but for $I(G)^2$ this
characteristic dependence has disappeared.

Proposition \ref{P.asymptoticRegularity} implies that the regularity of $(I(G) + (y^c))^h$ depends on the field for some $c \geq 3$ and for $h \gg 0$. However, computations with \textit{Macaulay2} show that already choosing $c=2$ produces a dependence in the regularity of the first four powers of $I(G) + (y^2)$. Polarizing $y^2$ as $x_{12}x_{13}$, the ideal $I(G) + (y^2)$ is transformed into the edge ideal $J = I(G) + (x_{12}x_{13})$, which corresponds to the disjoint union of the graph $G$ of Katzman's edge ideal and the edge $\{12,13\}$. Moreover, $\reg_k(J) = \reg_k(I(G) + (y^2))$ by \cite[Corollary 1.6.3 (c)]{HH11}.
We conjecture that $\reg_{\Q}(J^h) = 2h+1$ and $\reg_{\Z_2}(J^h) = 2h+2$ for every $h \geq 1$. If true, this would yield an edge ideal of a graph with $13$ vertices such that the regularity of all powers depends on the field (which is simpler than the graph of \cite[Remark 5.4]{MV21}).
\end{example}

\section{Spreading the characteristic dependence}

In this section, starting from a monomial ideal whose Betti numbers depend on the field, we show how to produce dependence in all powers of the ideal and in its Kodiyalam polynomials.

\subsection{Creating the dependence in all powers}
Recall that, any monomial ideal $I$ in a polynomial ring has a unique minimal system of monic monomial generators $G(I)$.

\begin{remark}\label{R.monomialTimesIdeal}
Let $I$ be a monomial ideal in $k[x_1,\dots,x_n,y_1,\dots,y_r]$, with generators in the variables $x_1,\dots,x_n$, and $w$ be a monomial of degree $d$ in the variables $y_1,\dots,y_r$. Then $\beta_{i} (wI) = \beta_{i} (I)$ for every $i \in \N$. In fact, we have 
$\beta_{i,j} (wI) = \beta_{i,j-d} (I)$, for every $i,j \in \N$. This easily follows from \cite[Theorem 2.1]{GPW99} and by observing that all elements of the lcm-lattice of $wI$ are obtained by multiplying the elements of the lcm-lattice of $I$ by $w$.
\end{remark}

\begin{lemma}\label{L.Splitting}
Let $I$ be a monomial ideal of $R=k[x_1,\dots,x_n,y_1,\dots,y_r,z]$ with generators in the variables $x_1,\dots,x_n$ and let $w$ be a monic monomial in the variables $y_1,\dots,y_r$. Fix $h \in \mathbb{N}_{>0}$. Then
\begin{enumerate}
    \item[$\mathrm{(1)}$]  $\beta_{i}((I+(w))^h)=\beta_{i}((zI+(w))^h)$, for every $i \in \mathbb{N}$.
    \item[$\mathrm{(2)}$] $(zI +(w))^h=(zI)^h+w(zI +(w))^{h-1}$ is a Betti splitting of $(zI +(w))^h$.
\end{enumerate}
\end{lemma}

\begin{proof}
$\mathrm{(1)}$  Since $R/((I+(w))^h+(z)) \cong R/((zI+(w))^h+(z-1))$, it is enough to show that the class of $z-1$ is regular over $R/((zI+(w))^h)$. Let $f \in R$ and assume that $(z-1)f \in (zI+(w))^h$. We may also assume that there are no monomials of $f$ in $(zI+(w))^h$. In fact, if $g \in (zI+(w))^h$ is a monomial of $f$, then $(z-1)f \in (zI+(w))^h$ if and only if $(z-1)(f-g) \in (zI+(w))^h$. 
    
Now, if $f \neq 0$, regarding $f$ as a polynomial in $z$, we consider $u$ to be the term of lowest degree (possibly zero) with respect to $z$. In $(z-1)f$ the term with lowest degree with respect to $z$ is $-u$, and it does not cancel with any other term of $(z-1)f$. Since $-u \notin (zI+(w))^h$, which is a monomial ideal, this means that $(z-1)f \notin (zI+(w))^h$, a contradiction. \\ 
$\mathrm{(2)}$ Let $m$ be a multidegree in the lcm-lattice of $w(zI)^h$. We claim that 
\begin{equation}\label{Eq.BettiSplitting}
\tag{$\star$}
\text{if } \beta_{i,m}(w(zI)^h) \neq 0, \text{ then }
\beta_{i,m}((zI)^h)=\beta_{i,m}(w(zI +(w))^{h-1})=0.
\end{equation}

Suppose that $G(I) = \{m_1,\dots,m_a\}$, where $m_i$ are monomials in the variables $x_1,\dots,x_n$, and the multidegree $m$ appears in the lcm-lattice of $w(zI)^h$.

Then, $m$ is not an element of the lcm-lattice of $I^h$. In fact, $wz^h$ is a factor of $m$ since all the generators of $w(zI)^h$ have the form $wz^h m_{i_1} \cdots m_{i_h}$, where $m_{i_j} \in G(I)$. Thus, $\beta_{i,m}((zI)^h)=0$. 

Assume that $m$ appears in the lcm-lattice of $w(zI +(w))^{h-1}$. Notice that, the ideal $w(zI +(w))^{h-1}$ is generated by monomials of the form $w^{s+1} z^t m_{i_1} \cdots m_{i_t}$, with $s+t=h-1$, where $m_{i_j} \in G(I)$. Hence, the atoms of the interval $(0,m)$ in the lcm-lattice of $w(zI +(w))^{h-1}$ are such that $s=0$, i.e., are generators of $w(zI)^{h-1}$. It follows that $z^{h-1}$ is the highest power of $z$ in the factorization of $m$, a contradiction. Thus, $\beta_{i,m}(w(I +(w))^{h-1})=0$.

To prove the statement, notice that $(zI +(w))^h=(zI)^h+w(zI +(w))^{h-1}$ and $G((zI)^h) \cap G(w(zI +(w))^{h-1})=\emptyset$. Moreover, $(zI)^h \cap w(zI +(w))^{h-1} = w(zI)^h$. From \eqref{Eq.BettiSplitting} it follows that all induced maps 
\[
\mathrm{Tor}_i^R(w(zI)^h,k)_m \rightarrow \mathrm{Tor}_i^R((zI)^h,k)_m \oplus \mathrm{Tor}_i^R( w(zI +(w))^{h-1},k)_m
\]
are zero, for every $i \in \N$ and every multidegree $m$.
\end{proof}

Starting from an ideal $I$ such that the Betti numbers of some power $I^h$ depend on the field, we add a monic monomial on new variables obtaining an ideal $J$ with the same property in all powers $J^q$ with $q \geq h$. This happens even if the higher powers of the original ideal $I$ have characteristic independent Betti numbers.

\begin{theorem}\label{T.Formula}
Let $I$ be a monomial ideal in $k[x_1,\dots,x_n,y_1,\dots,y_r]$, with generators in the variables $x_1,\dots,x_n$. Let $w$ be a monic monomial in the variables $y_1,\dots,y_r$ and fix $h \in \mathbb{N}_{>0}$. Then 
\begin{gather*}
\beta_{0}((I +(w))^h)=\sum_{\ell=1}^h \beta_{0}(I^{\ell})+1, \text{ and} \\
\beta_{i}((I +(w))^h)=\sum_{\ell=1}^h \left[\beta_{i}(I^{\ell})+\beta_{i-1}(I^{\ell}), \right] \text{    for every   } i \in \mathbb{N}_{>0}. \end{gather*}
In particular, if $\beta_i^{\Z_p}(I^h) \neq \beta_i^\Q(I^h)$ for some prime number $p$ and $i \geq 1$, then for every $q \geq h$
\[
\beta_i^{\Z_p}((I +(w)^q) \neq \beta_i^\Q((I +(w))^q).
\]
\end{theorem}

\begin{proof}
The formula for $\beta_{0}((I +(w))^h)$ follows immediately, because 
\[ 
(I +(w))^h=\sum_{\ell=1}^h w^{h-\ell}I^{\ell}+(w^h).
\]

Consider the ideal $I$ in the ring $R[z]$, where $z$ is a new variable. Fix now $i \geq 1$ and recall that $(zI)^h \cap w(zI +(w))^{h-1}=w(zI)^h$. By Lemma \ref{L.Splitting}(2), $(zI +(w))^h=(zI)^h+w(zI +(w))^{h-1}$ is a Betti splitting of $(zI +(w))^h$. Hence, by Remark \ref{R.monomialTimesIdeal} we obtain 
\[
\beta_{i}((zI +(w))^h)=\beta_{i}(I^h)+\beta_{i-1}(I^h)+\beta_{i}(w(zI +(w))^{h-1}).
\]
Observing that for $h=1$ we have $\beta_{i}((w))=0$ for $i \geq 1$, we get the formula by induction on $h$, by Remark \ref{R.monomialTimesIdeal} and Lemma \ref{L.Splitting}(1).

For the last part of the statement, suppose that $\beta^{\mathbb{Z}_p}_{i}(I^h) \neq \beta^{\mathbb{Q}}_{i}(I^h)$, for some prime $p$ and consider $q \geq h$. By the formula in the first part of the proof and inequality \eqref{Eq.universalCoefficientTheorem}, we have
\[
\beta_{i}^{\mathbb{Z}_p}((I+(w))^{q})-\beta_{i}^{\mathbb{Q}}((I +(w))^{q}) \geq \beta^{\mathbb{Z}_p}_{i}(I^h)-\beta^{\mathbb{Q}}_{i}(I^h) \geq 1. \qedhere 
\]
\end{proof}

\begin{example}
As seen in Example \ref{E.KatzmanEdgeIdeal}, Katzman's edge ideal $I(G)$ has characteristic-dependent Betti numbers but for $I(G)^2$ this dependence has disappeared.  Nevertheless, consider the graph $H$ obtained by adding a disjoint edge $\{y_1,y_2\}$ to $G$ and $I(H) = I(G) + (y_1y_2) \subseteq k[x_1,\dots,x_{12},y_1,y_2]$. Then, by Theorem \ref{T.Formula} the Betti numbers of $I(H)^h$ depend on the field for every $h \geq 1$. 
\end{example}

\begin{example}\label{E.nonDepEdgeIdealWithCharDepPowers}
We now construct an edge ideal $I(H)$ whose Betti numbers do not depend on the field and such that the Betti numbers of $I(H)^h$ depend on the field for every $h \geq 2$. Let $G$ be the graph of Example \ref{E.edgeIdealWithCharDepSquare} and consider the graph $H$ obtained by adding a disjoint edge $\{y_1,y_2\}$. Then, by Theorem \ref{T.Formula}, the Betti numbers of $I(H)^h$ depend on the field for every $h \geq 2$ and clearly do not depend for $h=1$.
\end{example}

For some reason, in the literature all explicit examples of monomial ideals whose resolution depends on the field have dependence in characteristic $2$. Clearly, it is well known that one can have
characteristic dependence in any characteristic. In the following we 
want to provide an explicit example for this dependence and use it
to propagate the dependence to powers. 

For every prime integer $p \geq 2$ there exist triangulable topological spaces with simplicial homology groups which are different with $\Q$ and $\Z_p$ coefficients (see for instance \cite[Theorem 40.9]{M84}). By the Stanley-Reisner correspondence and Hochster's formula \cite{H77}, this implies the existence of monomial ideals $I$ such that $\beta_i^{\Q}(I) \neq \beta_i^{\Z_p}(I)$, for some $i>0$. Here we present a class of such ideals coming from the so-called \textit{$p$-fold dunce cap}, which is a certain triangulation of a $2$-disk, where we identify its boundary in a suitable way, see \cite[Exercise 6, p.~41]{M84} and \cite[Example 5.11]{SW11}. For $p=2$, we obtain the real projective plane. We then extend the dependence to all powers by applying Theorem \ref{T.Formula}.

\begin{construction}\label{C.pfoldDunceCap}
Let $p \geq 2$ be a prime number. We are going to construct a $2$-dimensional triangulation $D_p$ of the so-called {\em p-fold dunce cap} with $2p+3$ vertices, $9p$ edges, and $7p-2$ facets.

Consider a regular $3p$-gon, with vertices labeled by cyclically repeating $1,2,3$ in clockwise order, see Figure \ref{F.3foldDunceCap} for a representation of the case $p=3$. Consider a regular $2p$-gon inside this, with vertices labeled by $4,\ldots,2p+3$. The facets of $D_p$ are:
\begin{itemize}
\item $\{2,k,k+1\}$, $\{1,2,k\}$, $\{1,3,k\}$, for every $4 \leq k \leq 2p+2$ even;
\item $\{3,k,k+1\}$, $\{2,3,k\}$, for every $5 \leq k \leq 2p+1$ odd;
\item $\{4,k,k+1\}$, for every  $5 \leq k \leq 2p+2$;
\item $\{2,3,2p+3\}$, $\{3,4,2p+3\}$.
\end{itemize}

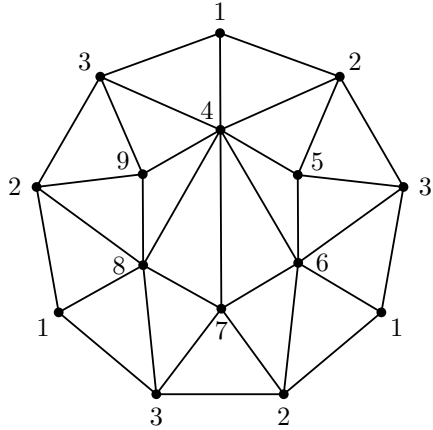
\begin{figure}[ht!]
\begin{tikzpicture}[scale=1.7]
\draw (0,0)-- (0.9965784671925104,0);
\draw (0.9965784671925104,0)-- (1.7600018641173611,0.6405882907917486);
\draw (1.7600018641173611,0.6405882907917486)-- (1.9330558988474436,1.6220264917679548);
\draw (1.9330558988474436,1.6220264917679548)-- (1.4347666652511888,2.485088761221226);
\draw (1.4347666652511888,2.485088761221226)-- (0.4982892335962555,2.8259386714056833);
\draw (0.4982892335962555,2.8259386714056833)-- (-0.4381881980586776,2.4850887612212262);
\draw (-0.4381881980586776,2.4850887612212262)-- (-0.9364774316549331,1.622026491767956);
\draw (-0.9364774316549331,1.622026491767956)-- (-0.7634233969248514,0.6405882907917494);
\draw (-0.7634233969248514,0.6405882907917494)-- (0,0);
\draw (-0.10623096204696426,1.7219623939801747)-- (-0.10271733566174435,1.010163643251805);
\draw (-0.10271733566174435,1.010163643251805)-- (0.5085806203848121,0.6691140736386733);
\draw (0.5085806203848121,0.6691140736386733)-- (1.1095871896428116,1.0279888480499286);
\draw (1.1095871896428116,1.0279888480499286)-- (1.1061904605131039,1.7161062760322694);
\draw (1.1061904605131039,1.7161062760322694)-- (0.5016702648698846,2.0690302523493846);
\draw (0.5016702648698846,2.0690302523493846)-- (-0.10623096204696426,1.7219623939801747);
\draw (0.4982892335962555,2.8259386714056833)-- (0.5016702648698846,2.0690302523493846);
\draw (0.5016702648698846,2.0690302523493846)-- (1.1095871896428116,1.0279888480499286);
\draw (0.5016702648698846,2.0690302523493846)-- (0.5085806203848121,0.6691140736386733);
\draw (0.5016702648698846,2.0690302523493846)-- (-0.10271733566174435,1.010163643251805);
\draw (0.5016702648698846,2.0690302523493846)-- (1.4347666652511888,2.485088761221226);
\draw (1.4347666652511888,2.485088761221226)-- (1.1061904605131039,1.7161062760322694);
\draw (1.1061904605131039,1.7161062760322694)-- (1.9330558988474436,1.6220264917679548);
\draw (1.9330558988474436,1.6220264917679548)-- (1.1095871896428116,1.0279888480499286);
\draw (1.1095871896428116,1.0279888480499286)-- (1.7600018641173611,0.6405882907917486);
\draw (1.1095871896428116,1.0279888480499286)-- (0.9965784671925104,0);
\draw (0.9965784671925104,0)-- (0.5085806203848121,0.6691140736386733);
\draw (0.5085806203848121,0.6691140736386733)-- (0,0);
\draw (0,0)-- (-0.10271733566174435,1.010163643251805);
\draw (-0.10271733566174435,1.010163643251805)-- (-0.7634233969248514,0.6405882907917494);
\draw (-0.9364774316549331,1.622026491767956)-- (-0.10271733566174435,1.010163643251805);
\draw (-0.9364774316549331,1.622026491767956)-- (-0.10623096204696426,1.7219623939801747);
\draw (-0.10623096204696426,1.7219623939801747)-- (-0.4381881980586776,2.4850887612212262);
\draw (-0.4381881980586776,2.4850887612212262)-- (0.5016702648698846,2.0690302523493846);
\node[label={below:{\small $3$}}] (a) at (0,0) {};
\node[label={below:{\small $2$}}] (b) at (0.9965784671925104,0) {};
\node[label={below right:{\small $1$}}] (c) at (1.7600018641173611,0.6405882907917486) {};
\node[label={right:{\small $3$}}] (d) at (1.9330558988474436,1.6220264917679548) {};
\node[label={above right:{\small $2$}}] (e) at (1.4347666652511888,2.485088761221226)  {};
\node[label={above:{\small $1$}}] (f) at (0.4982892335962555,2.8259386714056833) {};
\node[label={above left:{\small $3$}}] (g) at (-0.4381881980586776,2.4850887612212262) {};
\node[label={left:{\small $2$}}] (h) at (-0.9364774316549331,1.622026491767956) {};
\node[label={below left:{\small $1$}}] (i) at (-0.7634233969248514,0.6405882907917494) {};
\node[label={[label distance=1mm]115:{\small $4$}}] (j) at (0.5016702648698846,2.0690302523493846) {};
\node[label={below:{\small $7$}}] (k) at (0.5085806203848121,0.6691140736386733) {};
\node[label={[label distance=1mm]180:{\small $8$}}] (l) at (-0.10271733566174435,1.010163643251805) {};
\node[label={[label distance=1mm]0:{\small $6$}}] (m) at (1.1095871896428116,1.0279888480499286) {};
\node[label={[label distance=1mm]155:{\small $9$}}] (n) at (-0.10623096204696426,1.7219623939801747) {};
\node[label={[label distance=1mm]25:{\small $5$}}] (o) at (1.1061904605131039,1.7161062760322694) {};
\end{tikzpicture}
\caption{The triangulation $D_3$ representing a $3$-fold dunce cap}
\label{F.3foldDunceCap}
\end{figure}
\end{construction}

For instance, for $p=3$, the Stanley-Reisner ideal of $D_3$ is 
\begin{align*}
I_{D_3} = \ &(x_1 x_5, x_1 x_7, x_1 x_9, x_5 x_7, x_5 x_8, x_5 x_9, x_6 x_8, x_6 x_9, x_7 x_9, x_1 x_2 x_3, x_1 x_4 x_6, x_1 x_4 x_8, \\
&\ x_2 x_3 x_4, x_2 x_3 x_6, x_2 x_3 x_8, x_2 x_4 x_6, x_2 x_4 x_7, x_2 x_4 x_8, x_2 x_4 x_9, x_2 x_5 x_6, x_2 x_7 x_8, \\ 
&\  x_3 x_4 x_5, x_3 x_4 x_6, x_3 x_4 x_7, x_3 x_4 x_8, x_3 x_6 x_7, x_3 x_8 x_9) \subseteq k[x_1,\dots,x_9].
\end{align*}

Notice that, the Stanley-Reisner ideal $I_{D_p}$ is generated in degree $2$ and $3$. This is a consequence of the fact that $D_p$ is $2$-dimensional and hence any face has dimension $\leq 2$. Thus, a minimal non-face is of dimension $\leq 3$.

\begin{proposition}\label{P.pfoldDunceCap}
Let $p \geq 2$ be a prime number and $I_{D_p} \subseteq k[x_1, \ldots, x_{2p+3}]$ be the Stanley-Reisner ideal of the $p$-fold dunce cap in Construction \ref{C.pfoldDunceCap}. Then, $\pd_{\Z_p}(I_{D_p}) \neq \pd_\Q(I_{D_p})$.
\end{proposition}

\begin{proof}
By \cite[Example 5.11]{SW11}, we have $\widetilde{H}_1(D_p,\mathbb{Z})=\mathbb{Z}_p$ and $\widetilde{H}_0(D_p,\mathbb{Z})=\widetilde{H}_2(D_p,\mathbb{Z})=0$ whereas $\widetilde{H}_1(D_p,\mathbb{Z}_p)=\widetilde{H}_2(D_p,\mathbb{Z}_p)=\mathbb{Z}_p$ and  $\widetilde{H}_1(D_p,\mathbb{Q})=\widetilde{H}_2(D_p,\mathbb{Q})=0$ by the Universal Coefficients Theorem. 

Let $I_{D_p} \subseteq k[x_1,\dots,x_{2p+3}]$ be the Stanley-Reisner ideal of $D_p$. By the previous discussion, it follows that $\pd_{\Z_p}(I_{D_p}) - \pd_{\Q}(I_{D_p}) = 1$.
\end{proof}

In order to obtain an ideal having field dependent Betti numbers in finitely many different characteristics $p_1,\dots,p_r$, it is enough to consider the Stanley-Reisner ideal of various pairwise disjoint copies of $D_{p_1},D_{p_2},\ldots,D_{p_r}$.

\begin{corollary}\label{C.pfoldDunceCapCorollary}
For every prime number $p \geq 2$ there exists an edge ideal $I_p$ such that $\beta_i^{\mathbb{Z}_p}(I_p^h) \neq \beta_i^\mathbb{Q}(I_p^h)$ for some $i \geq 1$ and for every $h \geq 1$.
\end{corollary}

\begin{proof}
Let $p \geq 2$ be a prime number and consider the simplicial complex $D_p$ from Construction \ref{C.pfoldDunceCap}. Let $I_p$ be the edge ideal of the simplicial complex obtained either by taking  the Stanley-Reisner ideal of the barycentric subdivision of $D_p$ or performing on $D_p$ \cite[Construction 4.4]{DK14} by Dalili and Kummini. Notice that, $\beta^{\Z_p}_i (I_p) \neq \beta^\Q_i (I_p)$ for some $i$.

Then, by Theorem \ref{T.Formula}, it follows that
\[
\beta^{\Z_p}_i ((I_p + (y_1y_2))^h) \neq \beta^{\Q}_i ((I_p + (y_1y_2))^h),
\]
where $y_1,y_2$ are two new variables.
\end{proof}

\subsection{Kodiyalam polynomials}

As seen is Section \ref{S.asymptoticBehaviour}, Kodiyalam polynomials may depend on the characteristic of the field. In this subsection we show how to spread the dependence to high degree terms of these polynomials.

\begin{lemma}\label{L.spreadingDependenceToPowers}
Let $I$ be a monomial ideal in $k[x_1,\dots,x_n,y_1,\dots,y_r]$, with generators in the variables $x_1,\dots,x_n$. Let $w$ be a monic monomial in the variables $y_1,\dots,y_r$ and fix $h,i \in \mathbb{N}_{>0}$. Consider 
\[
B=\{1 \leq s \leq h: \beta_i^{\mathbb{Z}_p}(I^s) \neq \beta_i^\mathbb{Q}(I^s)\}.
\]
Then for every $q \geq h$ we have
\[
\beta_i^{\mathbb{Z}_p}((I+(w))^q) - \beta_i^\mathbb{Q}((I+(w))^q) \geq |B|. 
\]
\end{lemma}

\begin{proof}
First of all, if $B = \emptyset$, the claim follows by inequality \eqref{Eq.universalCoefficientTheorem}.
Suppose $B \neq \emptyset$ and let $s \in B$. Since $I^{s}$ is a monomial ideal, its minimal monomial generators are uniquely defined and independent of the field, and hence $i \geq 1$. By Theorem \ref{T.Formula}, it follows that 
\[
\beta_i^{k}((I+(w))^q)=\sum_{\ell=1}^q \left[\beta_i^k(I^{\ell})+\beta_{i-1}^k(I^{\ell}) \right].
\] 
Then 
\begin{align*}
\beta_i^{\mathbb{Z}_p}((I+(w))^q) - \beta_i^\mathbb{Q}((I+(w))^q) &= \sum_{\ell=1}^q \left[\beta_i^{\mathbb{Z}_p}(I^{\ell})+\beta_{i-1}^{\mathbb{Z}_p}(I^{\ell})-\beta_i^{\mathbb{Q}}(I^{\ell})-\beta_{i-1}^{\mathbb{Q}}(I^{\ell})\right]\\
& \geq \sum_{\ell=1}^h \left[\beta_{i-1}^{\mathbb{Z}_p}(I^{\ell})-\beta_{i-1}^{\mathbb{Q}}(I^{\ell})\right]+|B| \geq |B|,
\end{align*}
where the first inequality follows from \eqref{Eq.universalCoefficientTheorem}. 
\end{proof}

\begin{lemma}\label{L.SpreadingDependence2}
Let $I$ be a monomial ideal in $R=k[x_1,\dots,x_n]$ and $i, h, r \in \mathbb{N}$. Assume that $\beta^{\mathbb{Z}_p}_{i}(I^\ell) \neq \beta^{\mathbb Q}_{i}(I^\ell)$ for every $1 \leq \ell \leq h$. If $J=I+(y_1,\dots,y_{r+1})$ in $R[y_1,\dots,y_{r+1}]$, where $y_1,\dots,y_{r+1}$ are new variables, then $\beta^{\mathbb{Z}_p}_{i+a}(J^h)-\beta^{\mathbb Q}_{i+a}(J^h) \geq \binom{h+r}{r+1}$, for every $0 \leq a \leq r$.
\end{lemma}

\begin{proof}
We proceed by induction on $r \geq 0$. If $r=0$, the result follows by Lemma \ref{L.spreadingDependenceToPowers}. Let $r>0$. By induction, for the ideal $T=I+(y_1,\dots,y_r)$ we have $$\beta^{\mathbb{Z}_p}_{i+a}(T^{\ell})-\beta^{\mathbb Q}_{i+a}(T^{\ell}) \geq \binom{\ell+r-1}{r},$$ for every $1 \leq \ell \leq h$ and every $0 \leq a \leq r-1$. Fix $0 \leq a \leq r-1$. Since $J=T+(y_{r+1})$, by Theorem \ref{T.Formula} and inequality \eqref{Eq.universalCoefficientTheorem}, it follows that 
\[
\beta^{\mathbb{Z}_p}_{i+a}(J^{h})-\beta^{\mathbb Q}_{i+a}(J^{h}) \geq \sum_{\ell=1}^h \left[\beta^{\mathbb{Z}_p}_{i+a}(T^{\ell})-\beta^{\mathbb Q}_{i+a}(T^{\ell}) \right] \geq \sum_{\ell=1}^h \binom{\ell+r-1}{r}=\binom{h+r}{r+1}. 
\]

The statement for $a=r$ follows similarly by Theorem \ref{T.Formula} and using the fact that $\sum_{\ell = 1}^h \left[ \beta^{\mathbb{Z}_p}_{i+r}(T^\ell)-\beta^{\mathbb Q}_{i+r}(T^\ell) \right] \geq 0$ by \eqref{Eq.universalCoefficientTheorem}.
\end{proof}

Given a monomial ideal $I$ in $k[x_1,\dots,x_n]$, it is clear that $\beta^k_0(I)$ is the number of minimal generators of $I$, and hence it does not depend on $k$; moreover, the same holds for $\beta^k_1(I)$ by \cite[Corollary 5.3]{BH95}. Thus, $\mathfrak{P}_1^k(I)$ and $\mathfrak{P}_2^k(I)$ are independent of the characteristic of the field $k$. We show that this is not the case for $\mathfrak{P}_i^k(I)$ with $i \geq 3$.

\begin{theorem} \label{T.Kodiyalam}
For every $i \geq 3$ and for every $r \in \mathbb{N}$, there exists a monomial ideal $I$ such that all the Kodiyalam polynomials $\mathfrak P^k_{3}(I),\mathfrak P^k_{4}(I),\dots,\mathfrak P^k_{i+r}(I)$ have the coefficient at some degree $\geq r$ depending on the characteristic of $k$. 
\end{theorem}

\begin{proof}
Fix $i \geq 3$ and $r \in \mathbb{N}$. Let $A \subseteq k[x_1,\dots,x_{10}]$ be the monomial ideal introduced in Section \ref{S.KimuraTeraiYoshida} for which we know that $\beta_2^{\mathbb{Z}_2}(A^{\ell}) \neq \beta_2^{\mathbb{Q}}(A^{\ell})$, for every $\ell \in \mathbb{N}$. Set $I=A+(y_1, \dots, y_{r+i-2})$ in the polynomial ring $k[x_1,\dots,x_{10}, y_1, \dots, y_{r+i-2}]$. 
By Lemma \ref{L.SpreadingDependence2}, we have that $\beta^{\mathbb{Z}_2}_{2+a}(I^h)-\beta^{\mathbb Q}_{2+a}(I^h) \geq \binom{h+r+i-3}{r+i-2}$, for every $0 \leq a \leq r+i-3$.
If $h \geq (r+i-2)!$, we have
\[
\beta^{\Z_2}_{2+a}(I^h)-\beta^{\Q}_{2+a}(I^h) \geq 
\frac{(h+r+i-3) \cdots (h+1)h}{(r+i-2)!} \geq (h+r+i-3) \cdots (h+1) > h^{r+i-3} \geq h^{r}.
\]
This implies that $\mathfrak P^k_{3+a}(I)$ has a coefficient of degree at least $r$ that depends on the characteristic of $k$, for every $0 \leq a \leq i+r-3$.
\end{proof}

Notice that in this case the degree of the Kodiyalam polynomial goes up by one. Moreover, Theorem \ref{T.Kodiyalam} answers a question of Herzog and the fourth author, see the last paragraph of Section 1 in \cite{HW11}.

\section{Examples and questions}

In this section we collect open questions, conjectures and some interesting examples beyond monomial ideals.

\subsection{Questions and conjectures}

The following conjecture is based on numerous computer experiments.

\begin{conjecture}\label{C.coneEdgeIdeals}
Let $G$ be a connected graph on $n$ vertices, $I(G) \subseteq k[x_1,\dots,x_n]$ be its edge ideal and $J = I(G) + x_{n+1}(x_1,\dots,x_n) \subseteq k[x_1,\dots,x_n,x_{n+1}]$ be the edge ideal of the cone over $G$ from a new vertex $n+1$. If the Betti numbers of $I(G)$ depend on the field, then the same holds for $J^2$.
\end{conjecture}

It is easy to prove that the Betti numbers of $J$ depend on the field. In fact, $J$ is the Stanley-Reisner ideal of $\Delta \cup \{n+1\}$, where $\Delta$ is the Stanley-Reisner complex of $I(G)$. Thus, the Betti numbers of $J$ depend on the field by Hochster's formula.

However, the analog of Conjecture \ref{C.coneEdgeIdeals} for $J^3$ does not hold.

\begin{example}
Consider Katzman's edge ideal $I(G)$ in Example \ref{E.KatzmanEdgeIdeal} and $J = I(G) + x_{12}(x_1,\dots,x_{11}) \subseteq k[x_1,\dots,x_{12}]$. Then, the Betti numbers of $J$ and $J^2$ are different over $\Q$ and over $\Z_2$, but this does not happen for $J^3$.
\end{example}

We noticed that, if $I$ is a squarefree monomial ideal that is not generated only in degree two, then Conjecture \ref{C.coneEdgeIdeals} does not hold for $J^2$. This is the case for the ideal of the real projective plane \eqref{Eq.ProjectivePlane}.

\medskip

In Theorem \ref{T.KleinBottle} and Section \ref{S.KimuraTeraiYoshida} we presented examples of simplicial complexes $\Delta$ of dimension $\geq 2$ such that the Betti numbers of $I_\Delta^h$ depend on the field for every $h \geq 1$. On the other hand, even if the Betti numbers of the Stanley-Reisner ideal $I_\Delta$ of the real projective plane \eqref{Eq.ProjectivePlane} differ over $\Q$ and $\Z_2$, this is not the case for the first few powers of $I_\Delta$.

\begin{question}
Which topological spaces admit a triangulation $\Delta$ such that the Stanley-Reisner ideal $I_\Delta$ and all its powers have characteristic dependendent Betti numbers?
Can we find such simplicial complexes $\Delta$ of dimension $1$?
\end{question}

In Theorem \ref{T.KleinBottle} we saw an ideal such that the Betti numbers of all its powers depend on the field and in Example \ref{E.nonDepEdgeIdealWithCharDepPowers} we showed another ideal such that the same holds for all powers starting from the second one. It is then natural to ask the following:

\begin{question}
Given $h \geq 1$, can we find a monomial ideal $I_h$ such that the Betti numbers of $I_h^\ell$ do not depend on the field for $\ell < h$ and depend on the field for $\ell \geq h$?
\end{question}

Proposition \ref{P.KodiyalamPolynomials} shows that, given a monomial ideal $I$, for every $i$ there exists $h_i$ such that the Betti number $\beta_i(I^\ell)$ either depends on the field for every $\ell \geq h_i$ or it does not for every $\ell \geq h_i$. 

\begin{question}
Can we have any possible behaviour in the first powers?
\end{question}

In Theorem \ref{T.Kodiyalam} we saw that, given $i \geq 3$ and $r \in \N$, we can construct a monomial ideal $I$ for which the Kodiyalam polynomials $\mathfrak P_3^k(I), \dots, \mathfrak P_{i+r}^k(I)$ have a term of degree at least $r$ that depends on the characteristic of the field. However, we do not have control on $\deg(\mathfrak P_i^k(I))$.

\begin{question}
Is there a monomial ideal $I$ such that $\deg(\mathfrak P_i^k(I))$ or the coefficient of the top degree term of $\mathfrak P_i^k(I)$ depend on the field for some $i$?
\end{question}

In this paper we have compared the behaviour of Betti numbers of 
powers of monomial ideals when
taking coefficients over $\Z_p$ for a fixed prime $p$ and coefficients in
$\Q$. We also discussed extension to finite sets of primes. 
By Hochster's formula or the lcm-lattice formula it is obvious 
that the Betti numbers are constant for all but finitely many primes.
The situation for powers is less obvious. Even though we expect a positive answer, we
see no argument which could resolve the following question.

\begin{question}
  Let $I$ be a fixed monomial ideal and $i$ a fixed number. Is the set of
  sequences $(\beta_i^{\Z_p}(I^h))_{h \geq 1}$ where $p$ runs over all
  primes always finite ?
\end{question}

For example, we cannot rule out that there is a sequence of numbers $h(p)$, strictly increasing in $p$, such that 
$\beta_i^{\Z_p}(I^{h(p)}) \neq \beta_i^{\Z_q}(I^{h(p)})$ for
all primes $q \neq p$, while the $i$-the Betti numbers are 
identical otherwise.

\subsection{Binomial edge ideals} \label{bei}

In this paper we mainly dealt with monomial ideals. It makes sense to ask the same questions for other classes of combinatorially defined ideals, such as binomial edge ideals.

Given a field $k$ and a finite simple graph $G$ with vertex set $\{1, \dots, n\}$ and edge set $E(G)$, the \textit{binomial edge ideal} associated to $G$ and $k$ is the ideal
\[
J_G = (x_i y_j - x_j y_i : \{i,j\} \in E(G) )
\]
in the polynomial ring $k[x_1, \dots, x_n, y_1, \dots, y_n]$, where for simplicity we omit the $k$ in the notation $J_G$.
This class of ideals was introduced independently in \cite{HHHKR10} and \cite{O11} and has been extensively studied in the last decade. In \cite[Example 7.6]{BMS21}, the first three authors exhibit a graph $G$ such that the Betti numbers of $J_G$ depend on the field. However, it is still unknown whether the projective dimension or the regularity of $J_G$ may depend on the characteristic. 

In this section we provide some interesting examples for which the Betti numbers of some power of $J_G$ depend on the characteristic of the field. In particular, in the next example we show that the projective dimension of $J_G^3$ may be characteristic-dependent even if the Betti numbers of $J_G$ and $J_G^2$ are not.

\begin{example}
Consider the graphs $C$ and $D$ in Figure \ref{F.bei1}. \textit{Macaulay2} computations show that the Betti numbers of $J_{C}$, $J_{C}^2$, $J_{C}^3$, and of $J_{D}$, $J_{D}^2$ do not change when computed over $\mathbb{Q}$ or $\mathbb{Z}_2$. However, $\beta^{\mathbb{Z}_2}_5(J_{C}^4) \neq \beta^{\mathbb{Q}}_5(J_{C}^4)$ and $\pd_{\mathbb{Z}_2} (J_{D}^3) \neq \pd_{\mathbb{Q}} (J_{D}^3)$. 

\begin{figure}[ht!]
\begin{subfigure}[c]{0.45\textwidth}
\centering
\begin{tikzpicture}[scale=0.8]
\node[label={below:{\small $4$}}] (a) at (0,0) {};
\node[label={above:{\small $1$}}] (b) at (0,2) {};
\node[label={below:{\small $3$}}] (c) at (2,0) {};
\node[label={above:{\small $2$}}] (d) at (2,2) {};
\node[label={below:{\small $6$}}] (e) at (3,1) {};
\node[label={above:{\small $5$}}] (f) at (3,3) {};
\draw (2,0) -- (0,0) -- (0,2) -- (2,2) -- (2,0) -- (3,1) -- (3,3) -- (0,2);
\draw (3,3) -- (2,2);
\end{tikzpicture}
\caption{The graph $C$}
\end{subfigure}
\begin{subfigure}[c]{0.45\textwidth}
\centering
\begin{tikzpicture}[scale=0.8]
\node[label={below:{\small $1$}}] (a) at (0,0) {};
\node[label={above:{\small $2$}}] (b) at (0,2) {};
\node[label={below:{\small $5$}}] (c) at (2,0) {};
\node[label={above:{\small $3$}}] (d) at (1,3) {};
\node[label={above:{\small $4$}}] (e) at (2,2) {};
\node[label={below:{\small $6$}}] (f) at (1,1.5) {};
\draw (f) -- (a) -- (b) -- (d) -- (e) -- (c) -- (f) -- (b);
\draw (d) -- (f) -- (e);
\draw (a) -- (c);
\end{tikzpicture}

\caption{The graph $D$}
\end{subfigure}
\vspace{-4mm}
\caption{} \label{F.bei1}
\end{figure}
\end{example}
\vspace{-5mm}

Finally, we show some connected graphs with a small number of vertices whose binomial edge ideal has Betti numbers that change in several characteristics.

\begin{example}
Let $E$ and $F$ be the graphs in Figure \ref{F.bei2}.
Computations with \textit{Macaulay2} show that the Betti numbers of $J_{E}$ and $J_{E}^2$ are different in characteristic $0$, $2$, and $3$. For instance, $\beta_7^{\mathbb{Z}_2}(J_{E})=\beta_7^{\mathbb{Z}_3}(J_{E})+1=\beta_7^{\mathbb{Q}}(J_{E})+2$ and $\beta_7^{\mathbb{Z}_2}(J_{E}^2)=\beta_7^{\mathbb{Z}_3}(J_{E}^2)+3=\beta_7^{\mathbb{Q}}(J_{E}^2)+7$.
Moreover, the Betti numbers of $J_{F}$ are the same in these three characteristics, but they become different when we consider its square. Indeed, $\beta_3^{\mathbb{Z}_2}(J_{F}^2)-2=\beta_3^{\mathbb{Z}_3}(J_{F}^2)=\beta_3^{\mathbb{Q}}(J_{F}^2)$ and $\beta_5^{\mathbb{Z}_2}(J_{F}^2)=\beta_5^{\mathbb{Z}_3}(J_{F}^2)-2=\beta_5^{\mathbb{Q}}(J_{F}^2)$.

\begin{figure}[ht!]
\begin{subfigure}[c]{0.45\textwidth}
\centering
\begin{tikzpicture}[scale=0.6]
\node[label={below:{\small $5$}}] (a) at (0,0) {};
\node[label={left:{\small $6$}}] (b) at (0,2) {};
\node[label={above:{\small $1$}}] (b) at (0,4) {};
\node[label={below:{\small $4$}}] (d) at (4,0) {};
\node[label={above:{\small $2$}}] (e) at (2,4) {};
\node[label={above:{\small $3$}}] (f) at (4,4) {};
\node[label={below:{\small $7$}}] (f) at (2,2) {};
\node[label={below:{\small $8$}}] (c) at (0.5,2.5) {};
\node[label={right:{\small $9$}}] (c) at (1.5,3.5) {};
\draw (0,0) -- (0,2) -- (0,4) -- (2,4) -- (4,4) -- (4,0) -- (0,0) -- (2,2) -- (4,4);
\draw (4,0) -- (2,2) -- (0.5, 2.5) -- (0,4) -- (1.5, 3.5) -- (2,2);
\end{tikzpicture}
\caption{The graph $E$}
\end{subfigure}
\begin{subfigure}[c]{0.45\textwidth}
\centering
\begin{tikzpicture}[scale=0.8]
\node[label={above:{\small $1$}}] (a) at (-3,2) {};
\node[label={above:{\small $2$}}] (b) at (-1,2) {};
\node[label={above:{\small $3$}}] (c) at (0,2) {};
\node[label={above:{\small $4$}}] (d) at (1,2) {};
\node[label={above:{\small $5$}}] (e) at (3,2) {};
\node[label={below:{\small $8$}}] (f) at (-2,0) {};
\node[label={below:{\small $7$}}] (g) at (0,0) {};
\node[label={below:{\small $6$}}] (h) at (2,0) {};
\draw (a) -- (b) -- (c) -- (d) -- (e) -- (h) -- (g) -- (f) -- (a);
\draw (b) -- (g) -- (d);
\end{tikzpicture}
\caption{The graph $F$}
\end{subfigure}
\vspace{-3mm}
\caption{} \label{F.bei2}
\end{figure}
\end{example}

\vspace{-4mm}
\begin{question}
Let $G$ be a finite simple graph. In contrast to the case of monomial ideals, in numerous computer experiments we noticed that, if the Betti numbers of $J_{G}^h$ depend on the characteristic for some $h$, then the same holds for $J_{G}^{h'}$ for every $h' \geq h$. Is this always the case?
\end{question}

\end{document}